\documentclass[12pt]{amsart}

\usepackage{amsfonts,amsmath,amsthm,latexsym,amssymb,subfigure,calc,mathtools,graphicx,empheq,enumerate,epsfig,tikz,color}
\usetikzlibrary{arrows,decorations.pathmorphing,backgrounds,positioning,fit,decorations.pathreplacing}
\numberwithin{equation}{section}

\newtheorem{theorem}[equation]{Theorem}
\newtheorem{lemma}[equation]{Lemma}
\newtheorem{corollary}[equation]{Corollary}

\theoremstyle{definition}
\newtheorem{definition}[equation]{Definition}
\newtheorem{remark}[equation]{Remark}
\newtheorem{example}[equation]{Example}

\newcommand{\LE}{\hbox{\vbox{\hrule width 0.35em height 0.04 em}\vbox{\offinterlineskip\hbox{\kern -0.02em\vrule height 0.65em width 0.04em \hspace{0.1em}}}}}
\newcommand{\la}{\lambda}
\newcommand{\T}{\mathcal{T}}
\newcommand{\sym}{\mathfrak{S}}
\newcommand{\pt}{\mathcal{PT}}

\begin{document}
\title{On alternating signed permutations with the maximal number of fixed points}

\author{Kyoungsuk Park}
\address{Department of Mathematics \\ Ajou University \\ Suwon  443-749, Korea}
\email{bluemk00@ajou.ac.kr}

\thanks{This research was supported by Basic Science Research Program through the National Research Foundation of Korea(NRF) funded by the Ministry of Education(NRF2011-0012398).}



\begin{abstract}
  A conjecture by R. Stanley on a class of alternating permutations, which is proved by R. Chapman and L. Williams states that alternating permutations with the maximal number of fixed points is equidistributed with derangements. We extend this (type $A$) result to type $B$: We prove that various classes of alternating signed permutations with the maximal number of fixed points is equidistributed with certain types of derangements (of type $B$), respectively.
\end{abstract}

\maketitle

\section{Introduction}\label{intro}

Let $\sym_n$ be the symmetric group of all permutations of $[n]=\{1,2,\ldots,n\}$.
We let $\sigma=\sigma_1 \sigma_2 \cdots \sigma_n \in\sym_n$ denote the permutation with $\sigma(i)=\sigma_i$ for all $i\in[n]$.
A permutation $\sigma$ is \emph{alternating} if $\sigma_1>\sigma_2<\sigma_3>\sigma_4<\cdots\sigma_n$ and \emph{reverse alternating} if $\sigma_1<\sigma_2>\sigma_3<\sigma_4>\cdots\sigma_n$.
It is well-known that the number of alternating (and reverse alternating) permutations in $\sym_n$ is equal to the Euler number $E_n$.
Euler number $E_n$ counts many interesting objects; increasing binary trees, increasing 1-2 trees, simsun permutations, and orbits of the action of symmetric groups on the set of maximal chains in the poset of partitions. (See \cite{RS3}.)
Alternating permutations with the maximal number of fixed points form an interesting class in the set of alternating permutations, and they are considered in \cite{RS2}.
Let $d_k(n)$ and $d_k^\ast(n)$ be the number of alternating and reverse alternating permutations in $\sym_n$ with $k$ fixed points, respectively.
Then it is known (\cite{RS2}) that
\begin{equation}\label{eq1}
  \max\{k : d_{k}(n)\ne 0\} = \lceil n/2 \rceil, \,\,n \ge 4,
\end{equation}
\begin{equation}\label{eq2}
  \max\{k : d_{k}^{\ast}(n)\ne 0\} = \lceil (n+1)/2 \rceil, \,\,n \ge 5.
\end{equation}

R. Stanley made a conjecture on the equidistribution of the alternating (reverse alternating) permutations with the maximal number of fixed points and the derangements, respectively, whose proof is done by R. Chapman and L. Williams in \cite{RC-LW}:
$$d_{\lceil n/2\rceil}(n) = D_{\lfloor n/2 \rfloor} ,\,\, n \ge 4 $$
$$d^{\ast}_{\lceil (n+1)/2 \rceil}(n) = D_{\lfloor (n-1)/2 \rfloor} ,\,\,n \ge 5,$$
where $D_n$ is the number of derangements in $\sym_n$.
In particular, their proof used \emph{permutation tableaux} which are in naturally bijection with permutations and introduced by E. Steingr\'{i}msson and Williams in \cite{ES-LW}.

As a signed analog of alternating permutations, V. Arnold introduced a \emph{snake} in \cite{VA}, a signed permutation $\sigma=\sigma_1 \sigma_2 \cdots \sigma_n$ such that $0<\sigma_1>\sigma_2<\sigma_3>\sigma_4<\cdots\sigma_n$.
It is known by Springer that the number of snakes in $\sym_n^B$ equals to the \emph{Springer number} $S_n$ (\cite{TS}).
Springer number $S_n$ also counts many interesting objects (\cite{MJ});
cycle-alternating permutations, weighted Dyck or Mozkin paths, increasing trees, and increasing forests.
In this article, we first define possible types of alternating signed permutations including snakes and consider classes of alternating signed permutations with the maximal number of fixed points for each type.
We then extend the proof by Chapman and Williams to show that each class of alternating signed permutations with the maximal fixed points are in bijection with certain derangements of type $B$, using permutation tableaux of \emph{type $B$}.
The rest of this article is organized as follows:
In Section \ref{altb}, we introduce the notion of alternating \emph{signed} permutations and consider the possible maximal fixed points of them.
Permutation tableaux of type $B$ and their relation to signed permutations are explained in Section \ref{ptb}.
Finally, Section \ref{pr} is devoted to the proof of the main theorem.

\section{Alternating signed permutation}\label{altb}

In this section, we introduce necessary terms and notations for alternating signed permutations,
and state the main theorem of this article, whose proof is done in Section \ref{pr}.

A \emph{signed permutation $\sigma$ of length $n$} is a permutation on $[\pm n]=\{-n,\ldots,-1,1,\ldots,n\}$ satisfying $\sigma(-i)=-\sigma(i)$ for all $i\in[n]$.
We denote $\sym_{n}^{B}$ the group of signed permutations of length $n$ and consider $\sym_{n}$ as a subgroup of $\sym_{n}^{B}$.

Let $\sigma\in\sym_n^B$.
A positive integer $i$ is called a \emph{fixed point} of $\sigma$ if $\sigma(i)=i$.
We use a \emph{one line notation} for $\sigma$;
$\sigma=\sigma_1 \sigma_2 \cdots \sigma_n$ means that $\sigma(i)=\sigma_i$ for all $i\in[n]$.
We consider all possible types of alternating signed permutations as follows.

Note that a signed permutation of type $\mathrm{+DU}$ is a snake.
\begin{definition}
  Let $\sigma\in\sym_n^B$.
  Then a signed permutation $\sigma\in\sym_n^B$ is called an \emph{alternating signed permutation} if $\sigma_1>\sigma_2<\sigma_3>\sigma_4<\cdots$ or $\sigma_1<\sigma_2>\sigma_3<\sigma_4>\cdots$.
  We define four types of alternating signed permutations as follows;
  \begin{itemize}
    \item $\sigma$ is \emph{of type} $\mathrm{-DU}$ if $0>\sigma_1 >\sigma_2 <\sigma_3 >\cdots$,
    \item $\sigma$ is \emph{of type} $\mathrm{-UD}$ if $0>\sigma_1 <\sigma_2 >\sigma_3 <\cdots$,
    \item $\sigma$ is \emph{of type} $\mathrm{+DU}$ if $0<\sigma_1 >\sigma_2 <\sigma_3 >\cdots$, and
    \item $\sigma$ is \emph{of type} $\mathrm{+UD}$ if $0<\sigma_1 <\sigma_2 >\sigma_3 <\cdots$.
  \end{itemize}
\end{definition}

In the following theorem, we find the (possible) maximal number of fixed points for alternating signed permutations of each type.
\begin{theorem}\label{maxfxalt}
  Let $d_{k}^\mathrm{-DU}(n)$, $d_{k}^\mathrm{-UD}(n)$, $d_{k}^\mathrm{+DU}(n)$, and $d_{k}^\mathrm{+UD}(n)$ be the numbers of alternating signed permutations with $k$ fixed points of types $\mathrm{-DU}$, $\mathrm{-UD}$, $\mathrm{+DU}$, and $\mathrm{+DU}$, respectively.
  Then for $n\ge 1$,
    $$\max\{k:d_{k}^\mathrm{-DU}(n)\ne 0\} = \lceil (n-2)/2 \rceil,$$
    $$\max\{k:d_{k}^\mathrm{-UD}(n)\ne 0\} = \lceil (n-1)/2 \rceil,$$
    $$\max\{k:d_{k}^\mathrm{+DU}(n)\ne 0\} = \lceil n/2 \rceil,$$
    $$\max\{k:d_{k}^\mathrm{+UD}(n)\ne 0\} = \lceil (n+1)/2 \rceil.$$
\end{theorem}
\begin{proof}
  We give a proof only for the type $\mathrm{-DU}$, since the same line of proof works for other types also.
  Let $\sigma\in\sym_{n}^{B}$ be of type $\mathrm{-DU}$.
  Neither 1 nor 2 can be a fixed point of $\sigma$, because $0>\sigma_1 >\sigma_2 $.
  For $i\ge2$, $\sigma$ has at most one fixed point among $2i-1$ and $2i$ since $\sigma_{2i-1}>\sigma_{2i}$.
  Hence, $\sigma$ has at most $(n-2)/2$ fixed points if $n$ is even, and $(n-1)/2$ fixed points if $n$ is odd.

  If we let $\sigma_1=-1$, $\sigma_2=-2$ and $\sigma_{2j-1}=2j-1$, $\sigma_{2j}=-2j$ for $j\ge 2$, then $\sigma$ is an alternating signed permutation of type $\mathrm{-DU}$ with $\lceil{(n-2)/2}\rceil$ fixed points, and the proof is completed.
\end{proof}

\begin{remark}\label{rmkpm}
  An alternating permutation (reverse alternating permutation) in $\sym_n$ is an alternating signed permutation of type $\mathrm{+DU}$ (or type $\mathrm{+UD}$, respectively) in $\sym_n^B$.
  Because of Equation \ref{eq1}, Equation \ref{eq2}, and Theorem \ref{maxfxalt}, any alternating permutation with maximal number of fixed points in $\sym_n$ is an alternating signed permutation of type $\mathrm{+DU}$ with maximal number of fixed points in $\sym_n^B$.
\end{remark}

The following corollary is from Theorem \ref{maxfxalt}.
\begin{corollary}\label{maxcor}
  Let $\sigma\in\sym_n^B$ be an alternating signed permutation with the maximal number of fixed points.
  Then, there does not exist $i\in[n-1]$ satisfying both $\sigma(i)>i$ and $\sigma(i+1)>i+1$.
  Moreover, if $\sigma(i)>i$ for $i>1$, then $\sigma(i-1)<\sigma(i)>\sigma(i+1)$.
\end{corollary}
\begin{proof}
  Suppose that there exists $i\in[n-1]$ satisfying $\sigma(i)>i$ and $\sigma(i+1)>i+1$, that is, $\sigma(i)\ge i+1$ and $\sigma(i+1)\ge i+2$.
  If $\sigma(i)>\sigma(i+1)$, then $\sigma(i+2)>\sigma(i+1)\ge i+2$ hence $\sigma(i+2)>i+2$.
  Thus, there is no fixed point among $i$, $i+1$, and $i+2$.
  So, it is from Theorem \ref{maxfxalt} that $\sigma$ can not have the maximal number of fixed points.
  Similarly, if $\sigma(i)<\sigma(i+1)$, then there is no fixed point among $i-1$, $i$, and $i+1$.
  Therefore, $i\in[n-1]$ satisfying both $\sigma(i)>i$ and $\sigma(i+1)>i+1$ does not exist.

  It is easy to check that for $i>1$, if $\sigma(i)>i$ then $\sigma(i-1)<\sigma(i)>\sigma(i+1)$.
\end{proof}

A signed permutation $\sigma\in\sym_{n}^{B}$ is called a \emph{derangement of type $B$} if $\sigma$ has no fixed point.
The set of derangements of type $B$ in $\sym_{n}^{B}$ is denoted by $\mathfrak{D}_{n}^{B}$.
Our main theorem states that the alternating signed permutations of each type with the maximal number of fixed points are equidistributed with certain kind of derangements of type $B$. Note that, alternating signed permutations are in bijection with ``special" kinds of derangements.

\begin{theorem}\label{altder}
  Let $D_{n}^{B}$ be the number of derangements of type $B$ in $\sym_{n}^{B}$, $D_{n}^{-}=|\{\sigma\in\mathfrak{D}_{n}^{B}\,:\,\sigma(1)<0\}|$, and
  $D_{n}^{-D}=|\{\sigma\in\mathfrak{D}_{n}^{B}\,:\,0>\sigma(1)>\sigma(2)\}|$.
  Then for $n\ge 1 $,
  $$d_{\lceil (n-2)/2 \rceil}^\mathrm{-DU}(n)= D_{\lfloor (n+2)/2 \rfloor}^{-D},$$
  $$d_{\lceil (n-1)/2 \rceil}^\mathrm{-UD}(n)= D_{\lfloor (n+1)/2 \rfloor}^{-},$$
  $$d_{\lceil n/2 \rceil}^\mathrm{+DU}(n)= D_{\lfloor n/2 \rfloor}^{B},\mbox{ and}$$
  $$d_{\lceil (n+1)/2 \rceil}^\mathrm{+UD}(n)= D_{\lfloor (n-1)/2 \rfloor}^{B}.$$
\end{theorem}

We prove Theorem \ref{altder} using permutation tableaux of type $B$ in Section \ref{pr}.

\section{Permutation tableaux of type $B$}\label{ptb}

We introduce permutation tableaux of type $B$ and properties of them associated with signed permutations in this section.

For two positive integers $r\le n$, an \emph{$(r,n)$-diagram} is a left-justified array of boxes in an $r\times (n-r)$ rectangle with $\la_i$ boxes in the $i$th row, where $\la_1 \ge \la_2 \ge \cdots \ge \la_r \ge 0$.
A \emph{shifted $(r,n)$-diagram} is an $(r,n)$-diagram with the stair-shaped array of boxes added above, where the $j$th column from the left has $(n-r-j+1)$ additional boxes for $j\in[n-r]$.
We call the unique $(r,n)$-diagram in a shifted $(r,n)$-diagram the \emph{$(r,n)$-subdiagram}, and the $(n-r)$ topmost boxes in a shifted $(r,n)$-diagram \emph{diagonals}.

Rows and columns of a shifted $(r,n)$-diagram are labeled as follows:
From the northeast corner to the southwest corner of the $(r,n)$-subdiagram, follow the southeast border edges of the subdiagram and give labels $1,2,\ldots,n$ in order.
If a vertical edge earned the label $i$, then the corresponding to row is named as $\verb"row"\,i$ and labeled by $i$ on the left side of $\verb"row"\,i$,
and if a horizontal edge earned the label $j$, then the corresponding to column is named as $\verb"col"\,j$ and labeled by $j$ on the top of $\verb"col"\,j$, however, we usually omit the column labels.
For the remaining rows,
if the diagonal in a row is in $\verb"col"\,j$, then the row is named as $\verb"row"\,(-j)$ and labeled by $-j$ on the left side of $\verb"row"\,(-j)$.
See Figure \ref{ptb8}.
The \emph{length} of a row is defined as the number of boxes in the row.

A \emph{permutation tableau of type $B$} is a $(0,1)$-filling of a shifted $(r,n)$-diagram that satisfies the following conditions:
\begin{enumerate}
  \item Every column has at least one box with a 1.
  \item If a box has a 1 above it in the same column and a 1 to the left in the same row, then it has a 1. (This is called $\LE$-condition.)
  \item If a diagonal has a 0, then there is no 1 in the same row.
\end{enumerate}
We remark that permutation tableaux of type $A$ are permutation tableaux of type $B$ such that the diagonals are filled with all 0's.

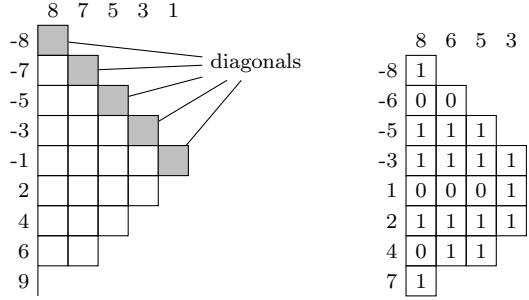
\begin{figure}[!ht]
  \begin{tikzpicture}\tiny
    \tikzstyle{Element} = [draw, minimum width=4mm, minimum height=4mm, node distance=4mm, inner sep=0pt]
      \node at (5.0,6.4) {8};
      \node at (5.4,6.4) {7};
      \node at (5.8,6.4) {5};
      \node at (6.2,6.4) {3};
      \node at (6.6,6.4) {1};
      \node (d) at (7.7,5.7) {diagonals};
      \node (d1) [Element,fill=lightgray] [label=left:-8] at (5.0,6.0) {};
      \node [Element] [label=left:-7] at (5.0,5.6) {};
      \node [Element] [label=left:-5] at (5.0,5.2) {};
      \node [Element] [label=left:-3] at (5.0,4.8) {};
      \node [Element] [label=left:-1] at (5.0,4.4) {};
      \node [Element] [label=left:2] at (5.0,4.0) {};
      \node [Element] [label=left:4] at (5.0,3.6) {};
      \node [Element] [label=left:6] at (5.0,3.2) {};
      \draw[-] (4.8,3.0) to (4.8,2.6);
      \node at (4.61,2.8) {9};

      \node (d2) [Element,fill=lightgray] at (5.4,5.6) {};
      \node [Element] at (5.4,5.2) {};
      \node [Element] at (5.4,4.8) {};
      \node [Element] at (5.4,4.4) {};
      \node [Element] at (5.4,4.0) {};
      \node [Element] at (5.4,3.6) {};
      \node [Element] at (5.4,3.2) {};

      \node (d3) [Element,fill=lightgray] at (5.8,5.2) {};
      \node [Element] at (5.8,4.8) {};
      \node [Element] at (5.8,4.4) {};
      \node [Element] at (5.8,4.0) {};
      \node [Element] at (5.8,3.6) {};

      \node (d4) [Element,fill=lightgray] at (6.2,4.8) {};
      \node [Element] at (6.2,4.4) {};
      \node [Element] at (6.2,4.0) {};

      \node (d5) [Element,fill=lightgray] at (6.6,4.4) {};

      \draw [-] (d) to (d1);
      \draw [-] (d) to (d2);
      \draw [-] (d) to (d3);
      \draw [-] (d) to (d4);
      \draw [-] (d) to (d5);
  \end{tikzpicture}\normalsize\quad\quad
  \begin{tikzpicture}\tiny
    \tikzstyle{Element} = [draw, minimum width=4mm, minimum height=4mm, node distance=4mm, inner sep=0pt]
      \node at (5,6) {8};
      \node at (5.4,6) {6};
      \node at (5.8,6) {5};
      \node at (6.2,6) {3};
      \node [Element] [label=left:-8] at (5.0,5.6) {1};
      \node [Element] [label=left:-6] at (5.0,5.2) {0};
      \node [Element] [label=left:-5] at (5.0,4.8) {1};
      \node [Element] [label=left:-3] at (5.0,4.4) {1};
      \node [Element] [label=left:1] at (5.0,4.0) {0};
      \node [Element] [label=left:2] at (5.0,3.6) {1};
      \node [Element] [label=left:4] at (5.0,3.2) {0};
      \node [Element] [label=left:7] at (5.0,2.8) {1};

      \node [Element] at (5.4,5.2) {0};
      \node [Element] at (5.4,4.8) {1};
      \node [Element] at (5.4,4.4) {1};
      \node [Element] at (5.4,4.0) {0};
      \node [Element] at (5.4,3.6) {1};
      \node [Element] at (5.4,3.2) {1};

      \node [Element] at (5.8,4.8) {1};
      \node [Element] at (5.8,4.4) {1};
      \node [Element] at (5.8,4.0) {0};
      \node [Element] at (5.8,3.6) {1};
      \node [Element] at (5.8,3.2) {1};

      \node [Element] at (6.2,4.4) {1};
      \node [Element] at (6.2,4.0) {1};
      \node [Element] at (6.2,3.6) {1};
  \end{tikzpicture}\normalsize
  \caption{A shifted $(4,9)$-diagram and a permutation tableau of type $B$.}\label{ptb8}
\end{figure}

The set of permutation tableaux of type $A$ and type $B$ of length $n$ are denoted by $\mathcal{PT}_{n}$ and $\mathcal{PT}_{n}^{B}$, respectively.
Steingr\'{i}msson and Williams defined a bijective zigzag map $\Phi:\pt_{n}\rightarrow\sym_{n}$ in \cite{ES-LW} and it was extended to a bijective zigzag map $\zeta:\pt_{n}^{B}\rightarrow\sym_{n}^{B}$ by S. Corteel and J. Kim in \cite{SC-JK}.

We give a definition of the zigzag map $\zeta$ as it was introduced in \cite{SC-KP}; in which zigzag map in this form played important role.

For a permutation tableau $\T$ of type $B$,
a \emph{zigzag path from} $\verb"row"\,i$ (or $\verb"col"\,i$) in $\T$, is the path starting from the left of $\verb"row"\,i$ (or the top of $\verb"col"\,i$, respectively), moving east (or south, respectively) until it meets the east edge of a row or the south edge of a column where it changes the direction to either south or east whenever it meets a $1$. Then the \emph{zigzag map} $\zeta: \pt_{n}^{B}\rightarrow\sym_{n}^{B}$ is defined as follows: For $\T\in\pt_{n}^{B}$ and $i\in[n]$,
\begin{itemize}
  \item[(1)] if $i$ is the label of a row of $\T$, then $\zeta(\T)(i)$ is the label of the row or the column that the zigzag path from $\verb"row"\,i$ ends,
  \item[(2)] if $-i$ is the label of a row of $\T$ and $\verb"row"\,(-i)$ has a $0$ in its diagonal, then $\zeta(\T)(i)$ is the label of the row or the column that the zigzag path from $\verb"col"\,i$ ends, and
  \item[(3)] if $-i$ is the label of a row of $\T$ and $\verb"row"\,(-i)$ has a $1$ in its diagonal, then $\zeta(\T)(-i)$ is the label of the row or the column that the zigzag path from $\verb"row"\,(-i)$ ends.
\end{itemize}
Zigzag map $\zeta$ is described in Figure~\ref{zigzagex}.

\tiny
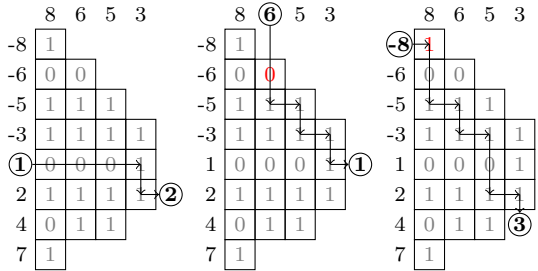
\begin{figure}[!ht]
  \begin{tikzpicture}
    \tikzstyle{Element} = [draw, minimum width=4mm, minimum height=4mm, node distance=4mm, inner sep=0pt]
    \tikzstyle{Vertex} = [draw,circle, minimum size=3mm, inner sep=0pt]
      \node at (5,6) {8};
      \node at (5.4,6) {6};
      \node at (5.8,6) {5};
      \node at (6.2,6) {3};
      \node [Element] [label=left:-8] at (5.0,5.6) {\textcolor{gray}{1}};
      \node [Element] [label=left:-6] at (5.0,5.2) {\textcolor{gray}{0}};
      \node [Element] [label=left:-5] at (5.0,4.8) {\textcolor{gray}{1}};
      \node [Element] [label=left:-3] at (5.0,4.4) {\textcolor{gray}{1}};
      \node [Element] at (5.0,4.0) {\textcolor{gray}{0}};
      \node [Element] [label=left:2] at (5.0,3.6) {\textcolor{gray}{1}};
      \node [Element] [label=left:4] at (5.0,3.2) {\textcolor{gray}{0}};
      \node [Element] [label=left:7] at (5.0,2.8) {\textcolor{gray}{1}};

      \node [Element] at (5.4,5.2) {\textcolor{gray}{0}};
      \node [Element] at (5.4,4.8) {\textcolor{gray}{1}};
      \node [Element] at (5.4,4.4) {\textcolor{gray}{1}};
      \node [Element] at (5.4,4.0) {\textcolor{gray}{0}};
      \node [Element] at (5.4,3.6) {\textcolor{gray}{1}};
      \node [Element] at (5.4,3.2) {\textcolor{gray}{1}};

      \node [Element] at (5.8,4.8) {\textcolor{gray}{1}};
      \node [Element] at (5.8,4.4) {\textcolor{gray}{1}};
      \node [Element] at (5.8,4.0) {\textcolor{gray}{0}};
      \node [Element] at (5.8,3.6) {\textcolor{gray}{1}};
      \node [Element] at (5.8,3.2) {\textcolor{gray}{1}};

      \node [Element] at (6.2,4.4) {\textcolor{gray}{1}};
      \node [Element] at (6.2,4.0) {\textcolor{gray}{1}};
      \node [Element] at (6.2,3.6) {\textcolor{gray}{1}};

      \node [Vertex] (0) at (4.6,4.0) {\textbf{1}};
      \node [Vertex] (1) at (6.6,3.6) {\textbf{2}};
      \path [->] (0) edge (6.2,4.0);
      \path [->] (6.2,4.0) edge (6.2,3.6);
      \path [->] (6.2,3.6) edge (1);
  \end{tikzpicture}
  \begin{tikzpicture}
    \tikzstyle{Element} = [draw, minimum width=4mm, minimum height=4mm, node distance=4mm, inner sep=0pt]
    \tikzstyle{Vertex} = [draw,circle, minimum size=3mm, inner sep=0pt]
      \node at (5,6) {8};
      \node at (5.8,6) {5};
      \node at (6.2,6) {3};
      \node [Element] [label=left:-8] at (5.0,5.6) {\textcolor{gray}{1}};
      \node [Element] [label=left:-6] at (5.0,5.2) {\textcolor{gray}{0}};
      \node [Element] [label=left:-5] at (5.0,4.8) {\textcolor{gray}{1}};
      \node [Element] [label=left:-3] at (5.0,4.4) {\textcolor{gray}{1}};
      \node [Element] [label=left:1] at (5.0,4.0) {\textcolor{gray}{0}};
      \node [Element] [label=left:2] at (5.0,3.6) {\textcolor{gray}{1}};
      \node [Element] [label=left:4] at (5.0,3.2) {\textcolor{gray}{0}};
      \node [Element] [label=left:7] at (5.0,2.8) {\textcolor{gray}{1}};

      \node [Element] at (5.4,5.2) {\textcolor{red}{0}};
      \node [Element] at (5.4,4.8) {\textcolor{gray}{1}};
      \node [Element] at (5.4,4.4) {\textcolor{gray}{1}};
      \node [Element] at (5.4,4.0) {\textcolor{gray}{0}};
      \node [Element] at (5.4,3.6) {\textcolor{gray}{1}};
      \node [Element] at (5.4,3.2) {\textcolor{gray}{1}};

      \node [Element] at (5.8,4.8) {\textcolor{gray}{1}};
      \node [Element] at (5.8,4.4) {\textcolor{gray}{1}};
      \node [Element] at (5.8,4.0) {\textcolor{gray}{0}};
      \node [Element] at (5.8,3.6) {\textcolor{gray}{1}};
      \node [Element] at (5.8,3.2) {\textcolor{gray}{1}};

      \node [Element] at (6.2,4.4) {\textcolor{gray}{1}};
      \node [Element] at (6.2,4.0) {\textcolor{gray}{1}};
      \node [Element] at (6.2,3.6) {\textcolor{gray}{1}};

      \node [Vertex] (0) at (5.4,6) {\textbf{6}};
      \node [Vertex] (1) at (6.6,4.0) {\textbf{1}};
      \path [->] (0) edge (5.4,4.8);
      \path [->] (5.4,4.8) edge (5.8,4.8);
      \path [->] (5.8,4.8) edge (5.8,4.4);
      \path [->] (5.8,4.4) edge (6.2,4.4);
      \path [->] (6.2,4.4) edge (6.2,4.0);
      \path [->] (6.2,4.0) edge (1);
  \end{tikzpicture}
  \begin{tikzpicture}
    \tikzstyle{Element} = [draw, minimum width=4mm, minimum height=4mm, node distance=4mm, inner sep=0pt]
    \tikzstyle{Vertex} = [draw,circle, minimum size=3mm, inner sep=0pt]
      \node at (5,6) {8};
      \node at (5.4,6) {6};
      \node at (5.8,6) {5};
      \node at (6.2,6) {3};
      \node [Element] at (5.0,5.6) {\textcolor{red}{1}};
      \node [Element] [label=left:-6] at (5.0,5.2) {\textcolor{gray}{0}};
      \node [Element] [label=left:-5] at (5.0,4.8) {\textcolor{gray}{1}};
      \node [Element] [label=left:-3] at (5.0,4.4) {\textcolor{gray}{1}};
      \node [Element] [label=left:1] at (5.0,4.0) {\textcolor{gray}{0}};
      \node [Element] [label=left:2] at (5.0,3.6) {\textcolor{gray}{1}};
      \node [Element] [label=left:4] at (5.0,3.2) {\textcolor{gray}{0}};
      \node [Element] [label=left:7] at (5.0,2.8) {\textcolor{gray}{1}};

      \node [Element] at (5.4,5.2) {\textcolor{gray}{0}};
      \node [Element] at (5.4,4.8) {\textcolor{gray}{1}};
      \node [Element] at (5.4,4.4) {\textcolor{gray}{1}};
      \node [Element] at (5.4,4.0) {\textcolor{gray}{0}};
      \node [Element] at (5.4,3.6) {\textcolor{gray}{1}};
      \node [Element] at (5.4,3.2) {\textcolor{gray}{1}};

      \node [Element] at (5.8,4.8) {\textcolor{gray}{1}};
      \node [Element] at (5.8,4.4) {\textcolor{gray}{1}};
      \node [Element] at (5.8,4.0) {\textcolor{gray}{0}};
      \node [Element] at (5.8,3.6) {\textcolor{gray}{1}};
      \node [Element] at (5.8,3.2) {\textcolor{gray}{1}};

      \node [Element] at (6.2,4.4) {\textcolor{gray}{1}};
      \node [Element] at (6.2,4.0) {\textcolor{gray}{1}};
      \node [Element] at (6.2,3.6) {\textcolor{gray}{1}};

      \node [Vertex] (0) at (4.6,5.6) {\textbf{-8}};
      \node [Vertex] (1) at (6.2,3.2) {\textbf{3}};
      \path [->] (0) edge (5.0,5.6);
      \path [->] (5.0,5.6) edge (5.0,4.8);
      \path [->] (5.0,4.8) edge (5.4,4.8);
      \path [->] (5.4,4.8) edge (5.4,4.4);
      \path [->] (5.4,4.4) edge (5.8,4.4);
      \path [->] (5.8,4.4) edge (5.8,3.6);
      \path [->] (5.8,3.6) edge (6.2,3.6);
      \path [->] (6.2,3.6) edge (1);
  \end{tikzpicture}
  \caption{Zigzag maps in permutation tableaux of type $B$.}\label{zigzagex}
\end{figure}\normalsize

For $\T\in\pt_{n}^{B}$, let the underlying shifted diagram of $\T$ be $(r,n)$-shifted diagram with shape $\la=(\la_1\ge\cdots\ge\la_r)$.
Let $\T^{+}$ be the sub-tableaux consisting of all positively labeled rows of $\T$ on the $(r,n)$-subdiagram with shape $\la$.
For $i\in[n]$, we call $\verb"row"\,i$ an \emph{empty row} if $\verb"row"\,i$ has only 0's in it.
The following lemma shows how some important statistics of $\sigma\in\sym_n^B$ can be recognized in the corresponding permutation tableau of type $B$.
\begin{lemma}\cite{ES-LW,SC-JK}\label{lem:bijec}
  Let $\T\in\pt_{n}^{B}$, $\sigma=\zeta(\T)$, and $lab(\T)=\{j\in[\pm n]\,:\,j\mbox{ is a row labeling of }\T\}$.
  Then, for $i\in[n]$,
  \begin{enumerate}
    \item $\sigma(i)\ge i$ if and only if $i\in lab(\T)$.
    \item $\sigma(i)=i$ if and only if $i\in lab(\T)$ and $\verb"row"\,i$ is empty.
    \item $\sigma(i)<i$ if and only if $(-i)\in lab(\T)$.
    \item $\sigma(i)<0$ if and only if $(-i)\in lab(\T)$ and $\verb"row"\,(-i)$ has a 1 in its diagonal.
  \end{enumerate}
\end{lemma}

The following lemma can be proved by the same idea of the proof of Lemma 6 in \cite{RC-LW}.
\begin{lemma}\label{lem:RCLW}
  Let $\sigma$ be an alternating signed permutation.
  \begin{enumerate}
    \item The number of consecutive positive entries which are fixed points of $\sigma$ is at most 2.
    \item If two consecutive positive entries $i$ and $i+1$ are fixed points of $\sigma$, then $i$ is even.
  \end{enumerate}
\end{lemma}

The following corollary is immediate from Corollary \ref{maxcor}, Lemma \ref{lem:bijec}, and Lemma \ref{lem:RCLW}.
\begin{corollary}\label{cor:RCLW}
  Let $\T$ be a permutation tableaux of type $B$ corresponding to an alternating signed permutation through $\zeta$. In $\T^{+}$,
  \begin{enumerate}
    \item the number of empty rows which have the same length is at most 2,
    \item if two consecutive rows $\verb"row"\,i$ and $\verb"row"\,(i+1)$ are empty, then $i$ is even,
    \item any two nonempty rows can not be consecutive.
  \end{enumerate}
\end{corollary}

The following lemmas are useful to prove Theorem \ref{altder} in Section \ref{pr}.
\begin{lemma}\label{consecrows}
  Let $\T\in\pt_n^B$, $\sigma=\zeta(\T)\in\sym_n^B$, and $i\in[n]$.
  \begin{enumerate}
    \item If $\T$ has two consecutive rows $\verb"row"\,i$ and $\verb"row"\,(i+1)$ (with the same length) such that $\verb"row"\,i$ is empty, then $\sigma(i)<\sigma(i+1)$.
    \item If $\T$ has two consecutive rows $\verb"row"\,i$ and $\verb"row"\,(i+1)$ (with the same length) such that $\verb"row"\,i$ is nonempty and $\verb"row"\,(i+1)$ is empty, then $\sigma(i)>\sigma(i+1)$.
    \item If $\T$ has two consecutive rows $\verb"row"\,i$ and $\verb"row"\,(i+2)$ (with lengths $j+1$ and $j$, respectively) such that $\verb"row"\,i$ is empty, then $\sigma(i)>\sigma(i+1)<\sigma(i+2)$.
  \end{enumerate}
\end{lemma}
\begin{proof}
  We first assume that $\T$ has two consecutive rows $\verb"row"\,i$ and $\verb"row"\,(i+1)$ with the same length such that $\verb"row"\,i$ is empty.
  Then, $\sigma(i)=i$ and $\sigma(i+1)\ge i+1$ by Lemma \ref{lem:bijec}.
  Hence, $\sigma(i)<\sigma(i+1)$.
  For the second statement, we assume that $\T$ has two consecutive rows $\verb"row"\,i$ and $\verb"row"\,(i+1)$ with the same length such that $\verb"row"\,i$ is nonempty and $\verb"row"\,(i+1)$ is empty.
  Then, $\sigma(i)>i$ and $\sigma(i+1)=i+1$ by Lemma \ref{lem:bijec}.
  So, we have $\sigma(i)>i+1=\sigma(i+1)$.
  For the last statement, we suppose that $\T$ has two consecutive rows $\verb"row"\,i$ and $\verb"row"\,(i+2)$ with lengths $j+1$, $j$ such that $\verb"row"\,i$ is empty.
  Then, $\T$ must have $\verb"col"\,(i+1)$.
  Thus, by Lemma \ref{lem:bijec}, $\sigma(i)=i$, $\sigma(i+1)<i+1$, and $\sigma(i+2)\ge i+2$, hence $\sigma(i)>\sigma(i+1)<\sigma(i+2)$.
\end{proof}

\begin{lemma}\label{leftmostones}
  Let $\T\in\pt_n^B$ and $\sigma=\zeta(\T)$.
  Let $\verb"row"\,(-i)$ and $\verb"row"\,(-j)$ be two consecutive rows of $\T$ with 1's in their diagonals for $i>j>0$.
  If the leftmost 1 of $\verb"row"\,(-i)$ is in $\verb"col"\,k$ and the leftmost 1 of $\verb"row"\,(-j)$ is in $\verb"col"\,l$ for $k>l$,
  then $\sigma(i)<\sigma(j)<0$.
\end{lemma}
\begin{proof}
  Let $\verb"row"\,(-i)$ and $\verb"row"\,(-j)$ be two consecutive rows of $\T$ with 1's in their diagonals and $i>j$, and the leftmost 1 of $\verb"row"\,(-i)$ is on the left of the leftmost 1 of $\verb"row"\,(-j)$.
  Then, $\sigma(i)<0$ and $\sigma(j)<0$ by Lemma \ref{lem:bijec}.
  It is easy to see that the two zigzag paths from $\verb"row"\,(-i)$ and $\verb"row"\,(-j)$ cross at the box in $\verb"row"\,(-j)$ and $\verb"col"\,k$, that is filled with 0, and they can not cross at another box filled with 0 because of the $\LE$-condition.
  This shows that the label of the southeast border edge that the zigzag path from $\verb"row"\,(-i)$ ends is greater than the one that the zigzag path from $\verb"row"\,(-j)$ ends.
  Therefore, $\sigma(i)<\sigma(j)<0$.
\end{proof}

\section{Proofs}\label{pr}

In this section, we prove the main theorem (Theorem \ref{altder}) using permutation tableaux of type $B$, extending the proof in \cite{RC-LW} for type $A$;
we give a detailed proof only for type $\mathrm{-DU}$, and only describe the bijective maps for other types.

We remark that the Algorithm for type $\mathrm{+DU}$ (or for type $\mathrm{+UD}$) restricted to the set of permutation tableaux of type $A$ induces the bijection from the set of derangements to the set of alternating (or reverse alternating, respectively) permutations with the maximal number of fixed points, which gives the proof by Chapman and Williams in \cite{RC-LW} for type $A$. (See Remark \ref{rmkpm}.)

Let $\mathcal{A}_{k}^\mathrm{-DU}(n)$, $\mathcal{A}_{k}^\mathrm{-UD}(n)$, $\mathcal{A}_{k}^\mathrm{+DU}(n)$, and $\mathcal{A}_{k}^\mathrm{+UD}(n)$ be the sets of permutation tableaux of type $B$ in $\pt_{n}^{B}$ corresponding to alternating signed permutations of types $\mathrm{-DU}$, $\mathrm{-UD}$, $\mathrm{+DU}$, and $\mathrm{+UD}$, respectively, with $k$ fixed points.
We also let
$$\mathcal{D}_n^B=\{\T\in\pt_n^B \,|\,\zeta(\T)\in\mathfrak{D}_n^B\},$$
$$\mathcal{D}_n^{-}=\{\T\in\mathcal{D}_n^B \,|\,\sigma(1)<0\mbox{ for }\sigma=\zeta(\T)\},$$
$$\mathcal{D}_n^{-D}=\{\T\in\mathcal{D}_n^{-}\,|\,\sigma(2)<\sigma(1)\mbox{ for }\sigma=\zeta(\T)\}.$$
Note that a permutation tableau $\T\in\mathcal{A}_{k}^{\ast}(n)$, for $\ast=\mathrm{-DU},\mathrm{-UD},\mathrm{+DU},\mathrm{+UD}$, has $k$ empty rows and has no three consecutive empty rows by Lemma \ref{lem:bijec} and Corollary \ref{cor:RCLW}.
Moreover, a permutation tableaux $\T\in\mathcal{D}_{n}^{B}$ has no empty row by Lemma \ref{lem:bijec}.

\subsection{Proof for type $\mathrm{-DU}$}

We first prove Theorem \ref{altder} for alternating signed permutations of type $\mathrm{-DU}$.
If $n$ is odd, then $n$ must be a fixed point. Hence we assume that $n$ is even for convention.
Let $n=2m-2$ and $\sigma\in\sym_n^B$ be an alternating signed permutation of type $\mathrm{-DU}$.
Then, by Theorem \ref{maxfxalt}, the maximal number of fixed points of $\sigma$ is $m-2$.

Let $\T\in\mathcal{A}_{k}^\mathrm{-DU}(n)$ or $\T\in\mathcal{D}_{n}^{-D}$.
Then, $0>\sigma(1)>\sigma(2)$ where $\sigma=\zeta(\T)$, and by Lemma \ref{lem:bijec} and Lemma \ref{leftmostones}, $\T$ has $\verb"row"\,(-1)$ and $\verb"row"\,(-2)$ with 1's in their diagonals, and the leftmost 1 of $\verb"row"\,(-2)$ is on the left of the leftmost 1 of $\verb"row"\,(-1)$.

We define a map $\Psi^\mathrm{-DU}:\mathcal{A}_{m-2}^\mathrm{-DU}(2m-2)\rightarrow\mathcal{D}_{m}^{-D}$.
For $\T\in\mathcal{A}_{m-2}^\mathrm{-DU}(2m-2)$, $\Psi^\mathrm{-DU}(\T)$ is obtained by deleting all empty rows from $\T$.
It is clear that $\Psi^\mathrm{-DU}(\T)\in\mathcal{D}_{m}^{B}$.
Moreover, since $\verb"row"\,(-1)$ and $\verb"row"\,(-2)$ of $\T$ are not empty, $\Psi^\mathrm{-DU}(\T)\in\mathcal{D}_{m}^{-D}$.

We now define a map $\Theta^\mathrm{-DU}:\mathcal{D}_{m}^{-D}\rightarrow\mathcal{A}_{m-2}^\mathrm{-DU}(2m-2)$ as follows:

\noindent \textbf{[Algorithm for $\Theta^\mathrm{-DU}$]}

\noindent Let $\T\in\mathcal{D}_{m}^{-D}$ and $\T^{+}$ has rows with lengths $\la_1\ge\la_2\ge\ldots\ge\la_r$, where $\la_i \ge 0$ for $i\in[r]$.
Then, $\Theta^\mathrm{-DU}(\T)$ is obtained by the following steps: (See Figure \ref{fig:algndu}.)
\begin{itemize}
  \item Insert $(m-r-\la_1 -2)$ empty rows between $\verb"row"\,(-1)$ and the first row of $\T^{+}$, whose lengths are $m-r-2,m-r-3,\ldots,\la_1 +1$.
  \item If $\la_1>0$, then for every $i\in[r-1]$, insert $(\la_i - \la_{i+1}+1)$ empty rows between the $i$th and $(i+1)$st rows of $\T^{+}$, whose lengths are $\la_i , \la_i , \la_i -1 , \la_i -2 , \ldots, \la_{i+1}+1$.
  \item If $\la_1>0$, then insert $(\la_r +1)$ empty rows after the $r$th row of $\T^{+}$, whose lengths are $\la_r , \la_r , \la_r -1 , \la_r -2 , \ldots, 2, 1$.
\end{itemize}

\begin{figure}[!ht]
  \begin{tikzpicture}
    \tikzstyle{Element} = [draw, minimum width=4mm, minimum height=4mm, inner sep=0pt]
      \node at (3.2,4.4) {$\T=$};
      \node at (10.3,4.4) {$=\Theta^{\mathrm{-DU}}(\T)$};
      \scriptsize
      \node at (4.6,6) {8};
      \node at (5,6) {6};
      \node at (5.4,6) {3};
      \node at (5.8,6) {2};
      \node at (6.2,6) {1};
      \node [Element] [label=left:-8] at (4.6,5.6) {0};
      \node [Element] [label=left:-6] at (4.6,5.2) {1};
      \node [Element] [label=left:-3] at (4.6,4.8) {0};
      \node [Element] [label=left:-2] at (4.6,4.4) {0};
      \node [Element] [label=left:-1] at (4.6,4.0) {0};
      \node [Element] [label=left:4] at (4.6,3.6) {1};
      \node [Element] [label=left:5] at (4.6,3.2) {0};
      \node [Element] [label=left:7] at (4.6,2.8) {1};

      \node [Element] at (5.0,5.2) {1};
      \node [Element] at (5.0,4.8) {0};
      \node [Element] at (5.0,4.4) {1};
      \node [Element] at (5.0,4.0) {0};
      \node [Element] at (5.0,3.6) {1};
      \node [Element] at (5.0,3.2) {1};

      \node [Element] at (5.4,4.8) {0};
      \node [Element] at (5.4,4.4) {0};
      \node [Element] at (5.4,4.0) {1};

      \node [Element] at (5.8,4.4) {1};
      \node [Element] at (5.8,4.0) {1};

      \node [Element] at (6.2,4.0) {1};
      \normalsize
      \node at (6.8,4.4) {$\longrightarrow$};
      \tiny
    \tikzstyle{Element} = [draw, minimum width=3mm, minimum height=3mm, inner sep=0pt]
      \node at (7.8,6.4) {14};
      \node at (8.1,6.4) {10};
      \node at (8.4,6.4) {4};
      \node at (8.7,6.4) {2};
      \node at (9.0,6.4) {1};
      \node[Element] [label=left:-14] at (7.8,6.1) {0};
      \node[Element] [label=left:-10] at (7.8,5.8) {1};
      \node[Element] [label=left:-4] at (7.8,5.5) {0};
      \node[Element] [label=left:-2] at (7.8,5.2) {0};
      \node[Element] [label=left:-1] at (7.8,4.9) {0};
      \node[Element] [label=left:3] at (7.8,4.6) {\textcolor{red}{0}};
      \node[Element] [label=left:5] at (7.8,4.3) {1};
      \node[Element] [label=left:6] at (7.8,4.0) {\textcolor{red}{0}};
      \node[Element] [label=left:7] at (7.8,3.7) {0};
      \node[Element] [label=left:8] at (7.8,3.4) {\textcolor{red}{0}};
      \node[Element] [label=left:9] at (7.8,3.1) {\textcolor{red}{0}};
      \node[Element] [label=left:11] at (7.8,2.8) {1};
      \node[Element] [label=left:12] at (7.8,2.5) {\textcolor{red}{0}};
      \node[Element] [label=left:13] at (7.8,2.2) {\textcolor{red}{0}};

      \node[Element] at (8.1,5.8) {1};
      \node[Element] at (8.1,5.5) {0};
      \node[Element] at (8.1,5.2) {1};
      \node[Element] at (8.1,4.9) {0};
      \node[Element] at (8.1,4.6) {\textcolor{red}{0}};
      \node[Element] at (8.1,4.3) {1};
      \node[Element] at (8.1,4.0) {\textcolor{red}{0}};
      \node[Element] at (8.1,3.7) {1};
      \node[Element] at (8.1,3.4) {\textcolor{red}{0}};
      \node[Element] at (8.1,3.1) {\textcolor{red}{0}};

      \node[Element] at (8.4,5.5) {0};
      \node[Element] at (8.4,5.2) {0};
      \node[Element] at (8.4,4.9) {1};
      \node[Element] at (8.4,4.6) {\textcolor{red}{0}};

      \node[Element] at (8.7,5.2) {1};
      \node[Element] at (8.7,4.9) {1};

      \node[Element] at (9.0,4.9) {1};

  \end{tikzpicture}\normalsize
  \caption{Algorithm for $\Theta^\mathrm{-DU}$; $m=8$, $r=3$, and $(\la_1,\la_2,\la_3)=(2,2,1)$}\label{fig:algndu}
\end{figure}

\begin{proof}[Proof of well-definedness of $\Theta^\mathrm{-DU}$]
  We claim that the map $\Theta^\mathrm{-DU}:\mathcal{D}_{m}^{-D}\rightarrow\mathcal{A}_{m-2}^\mathrm{-DU}(2m-2)$ is well-defined.
  Let $\T'=\Theta^\mathrm{-DU}(\T)$ and $\pi=\zeta(\T)$.
  It is clear that $0>\pi(1)>\pi(2)$ from the first insertion.

  We show that $\pi$ has $m-2$ fixed points.
  Let $t$ be the total number of empty rows that are inserted through Algorithm.
  If $\la_1=0$, then $r=0$ hence $t=m-2$.
  If $\la_1>0$, then
  $$t=(m-r-\la_1 -2)+(\la_1 -\la_2 +1)+\cdots+(\la_{r-1}-\la_{r}+1)+(\la_r +1) = m-2.$$

  Finally, we prove that $\pi$ is an alternating signed permutation.

  If the first insertion does not occur, then $\la_1=m-r-2$ hence $\T^{+}$ has $\verb"row"\,3$ and $\verb"row"\,3$ is nonempty.
  Then, $(\T')^{+}$ also has $\verb"row"\,3$ and $\verb"row"\,3$ is nonempty.
  Since $\pi(3)>0$, $\pi(2)<\pi(3)$.
  If the first insertion occurs, then $\T'$ has an empty row with length $m-r-2$ and this row must have the labeling 3.
  Thus, $\pi(2)<\pi(3)$.
  Moreover, since the length of each empty row of $\T'$ from the first insertion decreases by one,
  we have from Lemma \ref{consecrows} that
  $$\pi(3)>\pi(4)<\cdots>\pi(x-1)<\pi(x),$$
  where $\verb"row"\,x$ is the first nonempty row of $(\T')^{+}$.

  Assume that the second insertion occurs.
  If $\T^{+}$ has consecutive rows with the same length, then there is an empty row between them with the same length in $\T'$, namely, $\verb"row"\,x$.
  Then, by Lemma \ref{consecrows}, we have $\pi(x-1)>\pi(x)<\pi(x+1)$.
  If $\T^{+}$ has consecutive rows with different lengths $s_1$ and $s_2$, then there are empty rows between them with lengths $s_1, s_1, s_1 -1, s_1 -2, \ldots, s_2 +1$ in $\T'$.
  Let $\verb"row"\,x$ and $\verb"row"\,y$ be the nonempty rows with lengths $s_1$ and $s_2$ in $\T'$, respectively.
  Then, by Lemma \ref{consecrows},
  $$\pi(x)>\pi(x+1)<\pi(x+2)>\pi(x+3)<\pi(x+4)>\cdots<\pi(y-2)>\pi(y-1)<\pi(y).$$

  We now assume that the third insertion occurs.
  Let $\verb"row"\,x$ be the nonempty row with length $\la_r$ in $\T'$ from the downmost row of $\T$.
  Then, it is clear from \ref{consecrows} that
  $$\pi(x)>\pi(x+1)<\pi(x+2)>\cdots<\pi(2m-3)>\pi(2m-2).$$

  Therefore, $\pi$ is an alternating signed permutation of type $\mathrm{-DU}$ with $m-2$ fixed points,
  hence, $\Theta^\mathrm{-DU}(\T)\in\mathcal{A}_{m-2}^\mathrm{-DU}(2m-2)$.
\end{proof}

We claim that $\Psi^\mathrm{-DU}$ gives a bijection from $\mathcal{A}_{m-2}^\mathrm{-DU}(2m-2)$ to $\mathcal{D}_{m}^{-D}$, whose inverse is $\Theta^\mathrm{-DU}$.
It is easy to see that $\Psi^\mathrm{-DU}\circ\Theta^\mathrm{-DU}$ is the identity on $\mathcal{D}_m^{-D}$, and therefore $\Psi^\mathrm{-DU}$ is a surjection.
Thus, if we prove that $\Psi^\mathrm{-DU}$ is an injection, then the proof will be completed.

\begin{proof}[Proof of injectivity of $\Psi^\mathrm{-DU}$]
  Suppose we are trying to insert $m-2$ empty rows into $\T\in\mathcal{D}_{m}^{-D}$
  to get $\T'\in\mathcal{A}_{m-2}^\mathrm{+DU}(2m-2)$.
  To prove that $\Psi^\mathrm{-DU}$ is an injection, we show that there is no other way of insertion than the one given in Algorithm.

  Since $\T'$ has $\verb"row"\,(-1)$ and $\verb"row"\,(-2)$, we never insert empty rows with lengths $m-r$ and $m-r-1$.
  To make the number of empty rows be maximal, we must insert an empty row with length $m-r-2$.
  This row must have a labeling 3, hence $\T'$ has $\verb"row"\,3$.
  Since 3 is not even, by Corollary \ref{cor:RCLW}, we cannot insert more empty rows with length $m-r-2$ below $\verb"row"\,3$.
  So, we insert empty rows with lengths $m-r-2, m-r-3,\ldots$ above $\T^{+}$.

  Suppose that we can insert an empty row with length $\la_1$ just above the nonempty row with length $\la_1$ of $\T^{+}$, and let the empty row with length $\la_1$ have labeling $i$ in $\T'$.
  Then, $i-1$ is a labeling of a column and $i+1$ is a labeling of a nonempty row of length $\la_1$, hence $\sigma(i-1)<\sigma(i)<\sigma(i+1)$.
  This is a contradiction since $\sigma$ is alternating.
  Thus, the maximum number of empty rows that we can insert above $\T^{+}$ is $(m-r-\la_1 -2)$ and the lengths of the empty rows are $m-r-2, m-r-3, \ldots,\la_1 +1$.

  Same argument works for the remaining insertions, and this proves that $\Psi^\mathrm{-DU}$ is an injection.
  Thus, we are done for type $\mathrm{-DU}$.
\end{proof}

\begin{example}
  Let $\T\in\mathcal{D}_{8}^{B}$ be the permutation tableau of type $B$ with $(\la_1,\la_2,\la_3)=(2,2,1)$ in Figure \ref{fig:algndu}.
  Applying Algorithm for $\Theta^\mathrm{-DU}$ to $\T$ gives us the permutation tableau of type $B$ corresponding an alternating signed permutation $\sigma=-4,-5,3,2,11,6,10,8,9,-7,14,12,13$ of type $\mathrm{-DU}$.
  See Figure \ref{fig:algndu}.
\end{example}

\subsection{Type $\mathrm{-UD}$}

If $n$ is even, then $n$ must be a fixed point, hence, we assume that $n$ is odd for convention.
Let $n=2m-1$ and $\sigma\in\sym_n^B$ be an alternating signed permutation of type $\mathrm{-UD}$.
Then, by Theorem \ref{maxfxalt}, the maximal number of fixed points of $\sigma$ is $m-1$.

Let $\T\in\mathcal{A}_k^\mathrm{-UD}(n)$ or $\T\in\mathcal{D}_n^{-}$.
Then, $\sigma(1)<0$ where $\sigma=\zeta(\T)$, and $\T$ has $\verb"row"\,(-1)$ with 1 in its diagonal by Lemma \ref{lem:bijec}.

We define a map $\Psi^\mathrm{-UD}:\mathcal{A}_{m-1}^\mathrm{-UD}(2m-1)\rightarrow\mathcal{D}_{m}^{-D}$.
For $\T\in\mathcal{A}_{m-1}^\mathrm{-UD}(2m-1)$, $\Psi^\mathrm{-UD}(\T)$ is obtained by deleting all empty rows from $\T$.
It is clear that $\Psi^\mathrm{-UD}(\T)\in\mathcal{D}_{m}^B$.
Moreover, since $\verb"row"\,(-1)$ of $\T$ is not empty, $\Psi^\mathrm{-UD}(\T)\in\mathcal{D}_{m}^{-}$.

We now define a map $\Theta^\mathrm{-UD}:\mathcal{D}_{m}^{-}\rightarrow\mathcal{A}_{m-1}^\mathrm{-UD}(2m-1)$ as follows:

\noindent \textbf{[Algorithm for $\Theta^\mathrm{-UD}$]}

\noindent Let $\T\in\mathcal{D}_{m}^{-}$ and $\T^{+}$ has rows with lengths $\la_1\ge\la_2\ge\ldots\ge\la_r$, where $\la_i \ge 0$ for $i\in[r]$.
Then, $\Theta^\mathrm{-UD}(\T)$ is obtained by the following steps: (See Figure \ref{fig:algnud}.)
\begin{itemize}
  \item Insert $(m-r-\la_1 -1)$ empty rows between $\verb"row"\,(-1)$ and the first row of $\T^{+}$, whose lengths are $m-r-1,m-r-2,\ldots,\la_1 +1$.
  \item If $\la_1>0$, then for every $i\in[r-1]$, insert $(\la_i - \la_{i+1}+1)$ empty rows between the $i$th and $(i+1)$st rows of $\T^{+}$, whose lengths are $\la_i , \la_i , \la_i -1 , \la_i -2 , \ldots, \la_{i+1}+1$.
  \item If $\la_1>0$, then insert $(\la_r +1)$ empty rows after the $k$th row of $\T^{+}$, whose lengths are $\la_r , \la_r , \la_r -1 , \la_r -2 , \ldots, 2, 1$.
\end{itemize}

\begin{example}
  Let $\T\in\mathcal{D}_{7}^{-}$ be the permutation tableau of type $B$ with $(\la_1,\la_2,\la_3)=(2,2,1)$ in Figure \ref{fig:algnud}.
  Applying Algorithm for $\Theta^\mathrm{-UD}$ to $\T$ gives us the permutation tableau of type $B$ corresponding to the alternating signed permutation $\sigma=-4,2,1,10,5,9,7,8,-6,13,11,12,3$ of type $\mathrm{-UD}$.
  See Figure \ref{fig:algnud}.
\end{example}

\begin{figure}[!ht]
  \begin{tikzpicture}
    \tikzstyle{Element} = [draw, minimum width=4mm, minimum height=4mm, inner sep=0pt]
      \node at (3.2,4.4) {$\T=$};
      \node at (10.3,4.4) {$=\Theta^{\mathrm{-UD}}(\T)$};
      \scriptsize
      \node at (4.6,6) {7};
      \node at (5,6) {5};
      \node at (5.4,6) {2};
      \node at (5.8,6) {1};
      \node [Element] [label=left:-7] at (4.6,5.6) {0};
      \node [Element] [label=left:-5] at (4.6,5.2) {1};
      \node [Element] [label=left:-2] at (4.6,4.8) {0};
      \node [Element] [label=left:-1] at (4.6,4.4) {0};
      \node [Element] [label=left:3] at (4.6,4.0) {1};
      \node [Element] [label=left:4] at (4.6,3.6) {0};
      \node [Element] [label=left:6] at (4.6,3.2) {1};

      \node [Element] at (5.0,5.2) {1};
      \node [Element] at (5.0,4.8) {0};
      \node [Element] at (5.0,4.4) {1};
      \node [Element] at (5.0,4.0) {1};
      \node [Element] at (5.0,3.6) {1};

      \node [Element] at (5.4,4.8) {0};
      \node [Element] at (5.4,4.4) {1};

      \node [Element] at (5.8,4.4) {1};
      \normalsize
      \node at (6.6,4.4) {$\longrightarrow$};
      \tiny
    \tikzstyle{Element} = [draw, minimum width=3mm, minimum height=3mm, inner sep=0pt]
      \node at (7.8,6.4) {13};
      \node at (8.1,6.4) {9};
      \node at (8.4,6.4) {3};
      \node at (8.7,6.4) {1};
      \node[Element] [label=left:-13] at (7.8,6.1) {0};
      \node[Element] [label=left:-9] at (7.8,5.8) {1};
      \node[Element] [label=left:-3] at (7.8,5.5) {0};
      \node[Element] [label=left:-1] at (7.8,5.2) {0};
      \node[Element] [label=left:2] at (7.8,4.9) {\textcolor{red}{0}};
      \node[Element] [label=left:4] at (7.8,4.6) {1};
      \node[Element] [label=left:5] at (7.8,4.3) {\textcolor{red}{0}};
      \node[Element] [label=left:6] at (7.8,4.0) {0};
      \node[Element] [label=left:7] at (7.8,3.7) {\textcolor{red}{0}};
      \node[Element] [label=left:8] at (7.8,3.4) {\textcolor{red}{0}};
      \node[Element] [label=left:10] at (7.8,3.1) {1};
      \node[Element] [label=left:11] at (7.8,2.8) {\textcolor{red}{0}};
      \node[Element] [label=left:12] at (7.8,2.5) {\textcolor{red}{0}};

      \node[Element] at (8.1,5.8) {1};
      \node[Element] at (8.1,5.5) {0};
      \node[Element] at (8.1,5.2) {1};
      \node[Element] at (8.1,4.9) {\textcolor{red}{0}};
      \node[Element] at (8.1,4.6) {1};
      \node[Element] at (8.1,4.3) {\textcolor{red}{0}};
      \node[Element] at (8.1,4.0) {1};
      \node[Element] at (8.1,3.7) {\textcolor{red}{0}};
      \node[Element] at (8.1,3.4) {\textcolor{red}{0}};

      \node[Element] at (8.4,5.5) {0};
      \node[Element] at (8.4,5.2) {1};
      \node[Element] at (8.4,4.9) {\textcolor{red}{0}};

      \node[Element] at (8.7,5.2) {1};

  \end{tikzpicture}\normalsize
  \caption{Algorithm for $\Theta^\mathrm{-UD}$; $m=7$, $r=3$, and $(\la_1,\la_2,\la_3)=(2,2,1)$}\label{fig:algnud}
\end{figure}

\subsection{Type $\mathrm{+DU}$}

If $n$ is odd then $n$ must be a fixed point. Hence we assume that $n$ is even for convention.
Let $n=2m$ and $\sigma\in\sym_n^B$ be an alternating signed permutation of type $\mathrm{+DU}$.
Then, the maximal number of fixed points is $m$ by Theorem \ref{maxfxalt}.

Let $\T\in\mathcal{A}_{k}^\mathrm{+DU}(n)$.
Then, $\sigma(1)>0$ where $\sigma=\zeta(\T)$, and $\T$ has $\verb"row"\,i$ by Lemma \ref{lem:bijec}.
In order that $\sigma$ has the maximal number of fixed points, 1 must be fixed point of $\sigma$.
Hence, $\verb"row"\,1$ of $\T$ is empty.

We define a map $\Psi^\mathrm{+DU}:\mathcal{A}_{m}^\mathrm{+DU}(2m)\rightarrow\mathcal{D}_{m}^{B}$.
For $\T\in\mathcal{A}_{m}^\mathrm{+DU}(2m)$, $\Psi^\mathrm{+DU}(\T)$ is obtained by deleting all empty rows from $\T$.
Since $\verb"row"\,1$ of $\T$ is deleted, $\Psi^\mathrm{+DU}$ can have a row labeling either $1$ or $-1$.
Thus, $\Psi^\mathrm{+DU}(\T)\in\mathcal{D}_m^B$.

We now define a map $\Theta^\mathrm{+DU}:\mathcal{D}_{m}^{B}\rightarrow\mathcal{A}_{m}^\mathrm{+DU}(2m)$ as follows:

\noindent \textbf{[Algorithm for $\Theta^\mathrm{+DU}$]}

\noindent Let $\T\in\mathcal{D}_{m}^{B}$ and $\T^{+}$ has rows with lengths $\la_1\ge\la_2\ge\ldots\ge\la_r$, where $\la_i \ge 0$ for $i\in[r]$.
Then, $\Theta^\mathrm{+DU}(\T)$ is obtained by the following steps:
\begin{itemize}
  \item If $\T$ has $\verb"row"\,(-1)$, that is $\la_1 < m-r$, then insert $(m-r-\la_1)$ empty rows between $\verb"row"\,(-1)$ and the first row of $\T^{+}$, whose lengths are $m-r,m-r-1,\ldots,\la_1 +1$ from the top.
  \item If $\la_1>0$, then for every $i\in[r-1]$, insert $(\la_i - \la_{i+1}+1)$ empty rows between the $i$th and $(i+1)$st rows of $\T^{+}$, whose lengths are $\la_i , \la_i , \la_i -1 , \la_i -2 , \ldots, \la_{i+1}+1$.
  \item If $\la_1>0$, then insert $(\la_r +1)$ empty rows after the $k$th row of $\T^{+}$, whose lengths are $\la_r , \la_r , \la_r -1 , \la_r -2 , \ldots, 2, 1$.
\end{itemize}

\begin{example}
Let $\T\in\mathcal{D}_{7}^{B}$ be the permutation tableau of type $B$ with $(\la_1,\la_2,\la_3)=(2,2,1)$ in Figure \ref{fig:algpdu}.
Applying Algorithm for $\Theta^\mathrm{+DU}$ to $\T$ gives us the permutation tableau of type $B$ corresponding to the alternating signed permutation $\sigma=1,-5,3,2,11,6,10,8,9,-7,14,12,13,4$ of type $\mathrm{+DU}$.
See Figure \ref{fig:algpdu}.
\end{example}

\begin{figure}[!ht]
  \begin{tikzpicture}
    \tikzstyle{Element} = [draw, minimum width=4mm, minimum height=4mm, inner sep=0pt]
      \node at (3.6,4.4) {$\T=$};
      \node at (10.5,4.4) {$=\Theta^{\mathrm{+DU}}(\T)$};
      \scriptsize
      \node at (5,6) {7};
      \node at (5.4,6) {5};
      \node at (5.8,6) {2};
      \node at (6.2,6) {1};
      \node [Element] [label=left:-7] at (5.0,5.6) {0};
      \node [Element] [label=left:-5] at (5.0,5.2) {1};
      \node [Element] [label=left:-2] at (5.0,4.8) {0};
      \node [Element] [label=left:-1] at (5.0,4.4) {0};
      \node [Element] [label=left:3] at (5.0,4.0) {1};
      \node [Element] [label=left:4] at (5.0,3.6) {0};
      \node [Element] [label=left:6] at (5.0,3.2) {1};

      \node [Element] at (5.4,5.2) {1};
      \node [Element] at (5.4,4.8) {0};
      \node [Element] at (5.4,4.4) {1};
      \node [Element] at (5.4,4.0) {1};
      \node [Element] at (5.4,3.6) {1};

      \node [Element] at (5.8,4.8) {0};
      \node [Element] at (5.8,4.4) {1};

      \node [Element] at (6.2,4.4) {1};
      \normalsize
      \node at (7,4.4) {$\longrightarrow$};
      \tiny
    \tikzstyle{Element} = [draw, minimum width=3mm, minimum height=3mm, inner sep=0pt]
      \node at (8.2,6.7) {14};
      \node at (8.5,6.7) {10};
      \node at (8.8,6.7) {4};
      \node at (9.1,6.7) {2};
      \node[Element] [label=left:-14] at (8.2,6.4) {0};
      \node[Element] [label=left:-10] at (8.2,6.1) {1};
      \node[Element] [label=left:-4] at (8.2,5.8) {0};
      \node[Element] [label=left:-2] at (8.2,5.5) {0};
      \node[Element] [label=left:1] at (8.2,5.2) {\textcolor{red}{0}};
      \node[Element] [label=left:3] at (8.2,4.9) {\textcolor{red}{0}};
      \node[Element] [label=left:5] at (8.2,4.6) {1};
      \node[Element] [label=left:6] at (8.2,4.3) {\textcolor{red}{0}};
      \node[Element] [label=left:7] at (8.2,4.0) {0};
      \node[Element] [label=left:8] at (8.2,3.7) {\textcolor{red}{0}};
      \node[Element] [label=left:9] at (8.2,3.4) {\textcolor{red}{0}};
      \node[Element] [label=left:11] at (8.2,3.1) {1};
      \node[Element] [label=left:12] at (8.2,2.8) {\textcolor{red}{0}};
      \node[Element] [label=left:13] at (8.2,2.5) {\textcolor{red}{0}};

      \node[Element] at (8.5,6.1) {1};
      \node[Element] at (8.5,5.8) {0};
      \node[Element] at (8.5,5.5) {1};
      \node[Element] at (8.5,5.2) {\textcolor{red}{0}};
      \node[Element] at (8.5,4.9) {\textcolor{red}{0}};
      \node[Element] at (8.5,4.6) {1};
      \node[Element] at (8.5,4.3) {\textcolor{red}{0}};
      \node[Element] at (8.5,4.0) {1};
      \node[Element] at (8.5,3.7) {\textcolor{red}{0}};
      \node[Element] at (8.5,3.4) {\textcolor{red}{0}};

      \node[Element] at (8.8,5.8) {0};
      \node[Element] at (8.8,5.5) {1};
      \node[Element] at (8.8,5.2) {\textcolor{red}{0}};
      \node[Element] at (8.8,4.9) {\textcolor{red}{0}};

      \node[Element] at (9.1,5.5) {1};
      \node[Element] at (9.1,5.2) {\textcolor{red}{0}};

  \end{tikzpicture}\normalsize
  \caption{Algorithm for $\Theta^\mathrm{+DU}$; $m=7$, $r=3$, and $(\la_1,\la_2,\la_3)=(2,2,1)$}\label{fig:algpdu}
\end{figure}

\subsection{Type $\mathrm{+UD}$}

If $n$ is even, then $n$ must be a fixed point, hence, we assume that $n$ is odd for convention.
Let $n=2m+1$ and $\sigma\in\sym_n^B$ be an alternating signed permutation of type $\mathrm{+UD}$.
Then, by Theorem \ref{maxfxalt}, the maximal number of fixed points is $m+1$.

Let $\T\in\mathcal{A}_{k}^\mathrm{+UD}(n)$.
Then, $0<\sigma(1)<\sigma(2)$ where $\sigma=\zeta(\T)$, and $\T$ has $\verb"row"\,1$ and $\verb"row"\,2$ by Lemma \ref{lem:bijec}.
For $\sigma$, to have the maximal number of fixed points, both 1 and 2 must be fixed points.
Hence, $\verb"row"\,1$ and $\verb"row"\,2$ of $\T$ are empty.

We define a map $\Psi^\mathrm{+UD}:\mathcal{A}_{m+1}^\mathrm{+UD}(2m+1)\rightarrow\mathcal{D}_m^B$.
For $\T\in\mathcal{A}_{m+1}^\mathrm{+UD}(2m+1)$, $\Psi^\mathrm{+UD}(\T)$ is obtained by deleting all empty rows from $\T$.
Since $\verb"row"\,1$ and $\verb"row"\,2$ of $\T$ are deleted, $\Psi^\mathrm{+UD}(\T)$ can have the a labeling either $1$ or $-1$.
Thus, $\Psi^\mathrm{+UD}(\T)\in\mathcal{D}_m^B$.

We now define a map $\Theta^\mathrm{+UD}:\mathcal{D}_{m}^{B}\rightarrow\mathcal{A}_{m+1}^\mathrm{+UD}(2m+1)$ as follows:

\noindent \textbf{[Algorithm for $\Theta^\mathrm{+UD}$]}

\noindent Let $\T\in\mathcal{D}_{m}^{B}$ and $\T^{+}$ has rows with lengths $\la_1\ge\la_2\ge\ldots\ge\la_r$, where $\la_i \ge 0$ for $i\in[r]$.
Then, $\Theta^\mathrm{+UD}(\T)$ is obtained by the following steps: (See Figure \ref{fig:algpud}.)
\begin{itemize}
  \item If $\T$ has $\verb"row"\,(-1)$, that is $\la_1 < m-r$, then insert $(m-r-\la_1 +1)$ empty rows between $\verb"row"\,(-1)$ and the first row of $\T^{+}$, whose lengths are $m-r,m-r,m-r-1,m-r-2,\ldots,\la_1 +1$.
  \item If $\la_1>0$, then for every $i\in[r-1]$, insert $(\la_i - \la_{i+1}+1)$ empty rows between the $i$th and $(i+1)$st rows of $\T^{+}$, whose lengths are $\la_i , \la_i , \la_i -1 , \la_i -2 , \ldots, \la_{i+1}+1$.
  \item If $\la_1>0$, then insert $(\la_r +1)$ empty rows after the $k$th row of $\T^{+}$, whose lengths are $\la_r , \la_r , \la_r -1 , \la_r -2 , \ldots, 2, 1$.
\end{itemize}

\begin{example}
Let $\T\in\mathcal{D}_{7}^{B}$ be the permutation tableau of type $B$ with $(\la_1,\la_2,\la_3)=(2,2,1)$ in Figure \ref{fig:algpdu}.
Applying Algorithm for $\Theta^\mathrm{+UD}$ to $\T$ gives us the permutation tableau of type $B$ corresponding to the alternating signed permutation $\sigma=1,2,-6,4,3,12,7,11,9,10,-8,15,13,14,5$ of type $\mathrm{+UD}$.
See Figure \ref{fig:algpdu}.
\end{example}

\begin{figure}[!ht]
  \begin{tikzpicture}
    \tikzstyle{Element} = [draw, minimum width=4mm, minimum height=4mm, inner sep=0pt]
    \node at (3.2,4.4) {$\T=$};
      \node at (10.3,4.4) {$=\Theta^{\mathrm{+UD}}(\T)$};
      \scriptsize
      \node at (4.6,6) {7};
      \node at (5,6) {5};
      \node at (5.4,6) {2};
      \node at (5.8,6) {1};
      \node [Element] [label=left:-7] at (4.6,5.6) {0};
      \node [Element] [label=left:-5] at (4.6,5.2) {1};
      \node [Element] [label=left:-2] at (4.6,4.8) {0};
      \node [Element] [label=left:-1] at (4.6,4.4) {0};
      \node [Element] [label=left:3] at (4.6,4.0) {1};
      \node [Element] [label=left:4] at (4.6,3.6) {0};
      \node [Element] [label=left:6] at (4.6,3.2) {1};

      \node [Element] at (5.0,5.2) {1};
      \node [Element] at (5.0,4.8) {0};
      \node [Element] at (5.0,4.4) {1};
      \node [Element] at (5.0,4.0) {1};
      \node [Element] at (5.0,3.6) {1};

      \node [Element] at (5.4,4.8) {0};
      \node [Element] at (5.4,4.4) {1};

      \node [Element] at (5.8,4.4) {1};
      \normalsize
      \node at (6.6,4.4) {$\longrightarrow$};
      \tiny
    \tikzstyle{Element} = [draw, minimum width=3mm, minimum height=3mm, inner sep=0pt]
      \node at (7.8,7) {15};
      \node at (8.1,7) {11};
      \node at (8.4,7) {5};
      \node at (8.7,7) {3};
      \node[Element] [label=left:-15] at (7.8,6.7) {0};
      \node[Element] [label=left:-11] at (7.8,6.4) {1};
      \node[Element] [label=left:-5] at (7.8,6.1) {0};
      \node[Element] [label=left:-3] at (7.8,5.8) {0};
      \node[Element] [label=left:1] at (7.8,5.5) {\textcolor{red}{0}};
      \node[Element] [label=left:2] at (7.8,5.2) {\textcolor{red}{0}};
      \node[Element] [label=left:4] at (7.8,4.9) {\textcolor{red}{0}};
      \node[Element] [label=left:6] at (7.8,4.6) {1};
      \node[Element] [label=left:7] at (7.8,4.3) {\textcolor{red}{0}};
      \node[Element] [label=left:8] at (7.8,4.0) {0};
      \node[Element] [label=left:9] at (7.8,3.7) {\textcolor{red}{0}};
      \node[Element] [label=left:10] at (7.8,3.4) {\textcolor{red}{0}};
      \node[Element] [label=left:12] at (7.8,3.1) {1};
      \node[Element] [label=left:13] at (7.8,2.8) {\textcolor{red}{0}};
      \node[Element] [label=left:14] at (7.8,2.5) {\textcolor{red}{0}};

      \node[Element] at (8.1,6.4) {1};
      \node[Element] at (8.1,6.1) {0};
      \node[Element] at (8.1,5.8) {1};
      \node[Element] at (8.1,5.5) {\textcolor{red}{0}};
      \node[Element] at (8.1,5.2) {\textcolor{red}{0}};
      \node[Element] at (8.1,4.9) {\textcolor{red}{0}};
      \node[Element] at (8.1,4.6) {1};
      \node[Element] at (8.1,4.3) {\textcolor{red}{0}};
      \node[Element] at (8.1,4.0) {1};
      \node[Element] at (8.1,3.7) {\textcolor{red}{0}};
      \node[Element] at (8.1,3.4) {\textcolor{red}{0}};

      \node[Element] at (8.4,6.1) {0};
      \node[Element] at (8.4,5.8) {1};
      \node[Element] at (8.4,5.5) {\textcolor{red}{0}};
      \node[Element] at (8.4,5.2) {\textcolor{red}{0}};
      \node[Element] at (8.4,4.9) {\textcolor{red}{0}};

      \node[Element] at (8.7,5.8) {1};
      \node[Element] at (8.7,5.5) {\textcolor{red}{0}};
      \node[Element] at (8.7,5.2) {\textcolor{red}{0}};

  \end{tikzpicture}\normalsize
  \caption{Algorithm for $\Theta^\mathrm{+UD}$; $m=7$, $r=3$, and $(\la_1,\la_2,\la_3)=(2,2,1)$}\label{fig:algpud}
\end{figure}

\section*{Acknowledgement}
The author is grateful to Soojin Cho for valuable suggestions and helpful discussions.

\bibliographystyle{amsplain}
\bibliography{PTB}

\end{document}